\documentclass[a4paper,UKenglish]{lipics-v2019}

\usepackage[utf8]{inputenc}

\usepackage{parskip}
\usepackage{amsfonts}
\usepackage{amsmath,amssymb}
\usepackage{mathtools}
\usepackage{amsthm}

\usepackage{thmtools} 
\usepackage{thm-restate}  

\usepackage{microtype}

\usepackage{xcolor,pgf,tikz}
\usetikzlibrary{arrows,automata,shapes,decorations,calc,positioning}

\usepackage{xspace}

\usepackage{macros} 
\usepackage{enumitem}

\theoremstyle{definition}

\newtheorem*{defi}{Definition}

\newtheorem*{prop}{Proposition}

\DeclareMathOperator{\Log}{Log}

\setlist{nosep}

\title{How Fast Can You Escape a Compact Polytope?}

\author{Julian D'Costa}{Indian Institute of Science, Bangalore, India \and
Max Planck Institute for Software Systems,
  Saarland Informatics Campus, Germany}{julianrdcosta@gmail.com}{}{}

\author{Engel Lefaucheux}{Max Planck Institute for Software Systems, Saarland Informatics Campus, Germany}{elefauch@mpi-sws.org}{}{}

\author{Jo\"el Ouaknine}{Max Planck Institute for Software Systems,
  Saarland Informatics Campus, Germany \and Department of Computer
  Science, University of Oxford, UK}%
{joel@mpi-sws.org}{}{ERC grant AVS-ISS (648701) and DFG grant
389792660 as part of TRR 248 (see
https://perspicuous-computing.science).}

\author{James Worrell}{Department of Computer
  Science, University of Oxford, UK}{jbw@cs.ox.ac.uk}{}{EPSRC
  Fellowship EP/N008197/1.}

\authorrunning{J. D'Costa, E. Lefaucheux, J. Ouaknine, and J. Worrell} 

\Copyright{Julian D'Costa, Engel
  Lefaucheux, Jo\"el Ouaknine, James Worrell}

\ccsdesc[100]{Theory of computation → Timed and hybrid models}

\keywords{Continuous linear dynamical systems}

\nolinenumbers

\begin{document}

\maketitle

\begin{abstract}
  The Continuous Polytope Escape Problem (CPEP) asks whether every trajectory of a
  linear differential equation initialised within a convex polytope
 eventually escapes the polytope. We provide a polynomial-time algorithm to decide CPEP
  for compact polytopes. We also establish a quantitative uniform
  upper bound
  on the time required for every trajectory to escape the
  given polytope. In addition, we establish iteration bounds for termination
  of discrete linear loops via reduction to the continuous case.


\end{abstract}
\newpage

\section{Introduction}

In ambient space $\Reals^{d}$, a \emph{continuous linear
  dynamical system} is a trajectory $\myvector{x}(t)$, where $t$
ranges over the non-negative reals, defined by a differential equation
$\dot{\myvector{x}}(t)=f(\myvector{x}(t))$ in which the function
$f$ is \emph{affine} or \emph{linear}. If the initial point
$\myvector{x}(0)$ is given, the differential equation uniquely
defines the entire trajectory. (Linear) dynamical systems have been
extensively studied in Mathematics, Physics, and Engineering, and more
recently have played an increasingly important role in Computer
Science, notably in the modelling and analysis of cyber-physical
systems; two recent and authoritative textbooks on the subject
are~\cite{Alu15,Pla18}.

In the study of dynamical systems, particularly from the perspective
of control theory, considerable attention has been given to the study
of \emph{invariant sets}, \emph{i.e.}, subsets of $\Reals^{d}$ from which
no trajectory can escape; see, \emph{e.g.},
\cite{CastelanH92,BlondelT00,BM07,SDI08}. Our focus in the present
paper is on sets with the dual property that \emph{no trajectory
  remains trapped}. Such sets play a key role in analysing
\emph{liveness} properties in cyber-physical systems (see, for
instance,~\cite{Alu15}): discrete progress is ensured by
guaranteeing that all trajectories (\emph{i.e.}, from any initial starting
point) must eventually reach a point at which they `escape'
(temporarily or permanently) the set in question, thereby forcing a
discrete transition to take place.

More precisely, given an affine function
$f:\Reals^{d}\rightarrow \Reals^{d}$ and a convex polytope
$\mathcal{P}\subseteq\Reals^{d}$, both specified using rational
coefficients encoded in binary, we consider the \emph{Continuous Polytope
  Escape Problem (CPEP)} which asks whether, for all starting points
$\myvector{x}_0$ in $\mathcal{P}$, the corresponding
trajectory of the solution to the differential equation
\begin{equation*}
\begin{displaystyle} \begin{cases}
\dot{\myvector{x}}(t)=f(\myvector{x}(t)) \\
\myvector{x}(0)=\myvector{x}_{0}
\end{cases} \end{displaystyle}
\end{equation*}
eventually escapes $\mathcal{P}$.\footnote{By ``escaping''
  $\mathcal{P}$, we simply mean venturing outside of $\mathcal{P}$---we
  are unconcerned whether the trajectory might re-enter
  $\mathcal{P}$ at a later time or not.} 

CPEP was shown to be decidable in~\cite{OuakninePW17}, in which an
algorithm having complexity between $\NP$ and $\PSPACE$ was
exhibited. It is worth noting that, when the polytope $\mathcal{P}$ is
unbounded in space, the time taken for a given trajectory to escape
may be unboundedly large. For example, consider the unbounded one-dimensional polytope $\poly =
\{x\in \R \mid x \geq 1\}$ and differential equation
$\dot{{x}}(t)=-{x}(t)$.  For any starting point ${x}_0$, the
trajectory ${x}(t)= e^{-t} {x}_0$ converges to 0 and thus all trajectories 
eventually escape. However, the escape time is at least $\log(x_0)$ and
hence is not bounded over all initial points in $\poly$.  Even if the
polytope is bounded, there still need not be a uniform bound on the
escape time. For example, consider the polytope $\poly=(0,1]$ and the
  equation $\dot{{x}}(t)={x}(t)$.  Given an initial point ${x}_0$, the
  trajectory ${x}(t)= e^t {x}_0$ necessarily escapes $\poly$: but the
  escape time is at least $\log(1/x_0)$, which again is not bounded over
  $\poly$.  

  \textbf{Main contributions.} We show that, for \emph{compact}
  (\emph{i.e.}, closed and bounded) polytopes, CPEP is decidable in
  polynomial time.  Moreover, we show how to calculate uniform
  escape-time upper bounds; these bounds are exponential in the bit
  size of the descriptions of the differential equation and of the
  polytope, and doubly exponential in the ambient dimension.  In the
  case of differential equations specified by invertible or
  diagonalisable matrices, we have singly exponential bounds.

  In comparing the above with the results from~\cite{OuakninePW17}, we
  note both a substantial improvement in complexity (from $\PSPACE$ to
  $\PTIME$) as well as the production of explicit uniform bounds on
  escape times. It is worth pointing out that the mathematical
  approach pursued in~\cite{OuakninePW17} is non-effective, and
  therefore does not appear capable of yielding any quantitative
  escape-time bounds. The new constructive techniques used in the
  present paper, which originate mainly from linear algebra and
  algebraic number theory, are applicable owing to the fact that we
  focus our attention on \emph{compact} polytopes. In practice, of
  course, this is usually not a burdensome restriction; in most
  cyber-physical systems applications, for instance, all relevant
  polytopes will be compact (see, \emph{e.g.},~\cite{Alu15}).

 Another interesting observation is that the seemingly closely related
 question of whether a given single trajectory of a linear dynamical
 system escapes a compact polytope appears to be vastly more
 challenging and is not known to be decidable; see, in
 particular,~\cite{ContinuousSkolem,COW16b:LICS16,ContinuousSkolem3}. However,
 whether a given trajectory eventually hits a given single point is
 known as the \emph{Continuous Orbit Problem} and can be decided in
 polynomial time~\cite{Hainry08}.

 Finally, we also consider in the present paper a discrete analogue of
 CPEP for discrete-time linear dynamical systems, namely the
 \emph{Discrete Polytope Escape Problem (DPEP)}. This consists in
 deciding, given an affine function
$f:\Reals^{d}\rightarrow \Reals^{d}$ and a convex polytope
$\poly\subseteq\Reals^{d}$, whether for all initial points 
$\myvector{x}_{0} \in \poly$, the sequence
$(\myvector{x}_n)_{n\in \N}$ defined by the initial point and the
recurrence $\myvector{x}_{n+1} = f(\myvector{x}_n)$ eventually escapes
$\poly$. This problem---phrased as ``termination of linear programs''
over the reals and the rationals respectively---was already studied
and shown decidable in the seminal
papers~\cite{Braverman2006,Tiw04}, albeit with no complexity bounds
nor upper bounds on the number of iterations required to escape.
By leveraging our results on CPEP, we are able to show that, for
\emph{compact} polytopes, DPEP is decidable in polynomial time, and
moreover we derive upper bounds on the number of iterations that are
singly exponential in the bit size of the problem description and
doubly exponential in the ambient dimension.

\section{Preliminaries}
	
\subsection{The Continuous Polytope Escape Problem}

As noted in the previous section,	
the Continuous Polytope Escape Problem (CPEP) for continuous linear dynamical systems
consists in deciding, given an affine function
$f:\Reals^{d}\rightarrow \Reals^{d}$ and a convex polytope
$\poly\subseteq\Reals^{d}$, whether there exists an initial point
$\myvector{x}_{0} \in \poly$ for which the trajectory of the unique
solution of the differential equation
$\dot{\myvector{x}}(t)=f(\myvector{x}(t)),\myvector{x}(0)=\myvector{x}_{0},
t\geq 0$,
is entirely contained in $\poly$. For $T\in \R\cup \{\infty\}$, we denote by 
$X(T)$ the set $\{\vec{x}(t) \mid t \in \mathbb{R}_{\geq 0}, t \leq T\}$.
A starting point $\myvector{x}_{0}\in\poly$ is said to be a
\emph{fixed point} if for all $t\geq 0$,
$\myvector{x}(t)=\myvector{x}_0$, and it is \emph{trapped} if the
trajectory of $\myvector{x}(t)$ is contained in $\poly$ (\emph{i.e.},
$X(\infty)\subseteq \poly$); thus solving the CPEP amounts to deciding whether there is a trapped
point.

We will represent a $d$-dimensional instance of the CPEP by a triple
$(A,B,\myvector{c})$, where $A\in\R^{d\times d}$ represents the linear
function $f_A: \myvector{x} \mapsto A\myvector{x}$
\footnote{We remark that by increasing the dimension by one, the
  general CPEP can be reduced to the homogeneous case, in which the
  function $f$ is linear.}  and $B\in\R^{n\times d}$,
$\myvector{c}\in\R^{n}$ represent the polytope $\poly_{B,\myvector{c}} =
\{\myvector{x}\in \R^d \mid B\myvector{x}\leq \myvector{c}\}$.
Given such an instance and an initial point $\myvector{x}_0$, the
solution of the 
differential equation is $\myvector{x}(t) = \exp(At)\myvector{x}_0 \in \R^d$.
For the computation of bounds, we assume that
all the coefficients of $A$, $B$ and $\myvector{c}$ are rational and encoded in binary. 
The decidability results and escape bounds computed in this paper can be adapted to the case of 
algebraic coefficients, but we don't pursue this here.


	

Decidability of the CPEP was shown in~\cite{OuakninePW17}.  In this
paper we are interested in the following problem: given a positive
instance of CPEP (\emph{i.e.}, one in which every trajectory escapes),
compute an upper bound on the time to escape that holds uniformly over
all initial points in the polytope.  In other words, we wish to
compute $T\in \R_{\geq 0}$ such that for all points $\myvector{x}_{0}
\in \poly$ there exists $t_0\in \R$ such that $t_0\leq T$ and
$\myvector{x}(t_0)\not \in \poly$.  We call such a $T$ an
\emph{escape-time bound}.

As noted in the Introduction,
such an escape-time bound need not exist in general. In the remainder
of this paper, we therefore
restrict our attention to \emph{compact} 
polytopes.

   \subsection{Jordan Normal Forms}
   

   Let $A \in \mathbb{Q}^{d \times d}$ be a square matrix with rational
   entries.  The \emph{minimal polynomial} of $A$ is the unique monic
   polynomial $m(x) \in \mathbb{Q}[x]$ of least degree such that
   $m(A)=0$.  By the Cayley-Hamilton Theorem, the degree of $m$ is at most
   the dimension of $A$. The set $\sigma(A)$ of eigenvalues of $A$ is the set of
   roots of $m$.  The \emph{index} of an eigenvalue $\lambda$, denoted
   by $\nu(\lambda)$, is defined as its multiplicity as a root of $m$.
%

For each eigenvalue $\lambda$ of $A$ we denote by
$\mathcal{V}_{\lambda}$ the subspace of $\Complex^d$ spanned by the
set of generalised eigenvectors associated with $\lambda$. We also
denote by $\mathcal{V}^{r}$ the subspace of $\Complex^d$ spanned by
the set of generalised eigenvectors associated with some real
eigenvalue; we likewise denote by $\mathcal{V}^c$ the subspace of
$\Complex^d$ spanned by the set of generalised eigenvectors associated
with some non-real eigenvalue.

   
   
   It is well known that each vector $\myvector{v}\in\Complex^{d}$
   can be written uniquely as
   $\myvector{v}=\displaystyle{
   	\sum\limits_{\lambda\in\sigma(A)}\myvector{v}_{\lambda}}$,
   where $\myvector{v}_{\lambda}\in\mathcal{V}_{\lambda}$.
   It follows that $\myvector{v}$ can also be uniquely written as
   $\myvector{v}=\myvector{v}^{r}+\myvector{v}^{c}$, where
   $\myvector{v}^{r} \in\mathcal{V}^{r}$ and
   $\myvector{v}^{c} \in\mathcal{V}^{c}$.
   Moreover, we can write any matrix $A$ as $A=Q^{-1}JQ$ for some
   invertible matrix $Q$ and block diagonal Jordan matrix
   $J=\diag{(J_{1},\ldots,J_{N})}$, with each block $J_{i}$, associated to 
   the eigenvalue $\lambda_i$ having the following form:
   \begin{equation*}
   \begin{pmatrix}
   \lambda_i	&&	1		&&	0		&&	\cdots	&&	0		\\
   0		&&	\lambda_i	&&	1		&&	\cdots	&&	0		\\
   \vdots	&&	\vdots	&&	\vdots	&&	\ddots	&&	\vdots	\\
   0		&&	0		&&	0		&&	\cdots	&&	1		\\
   0		&&	0		&&	0		&&	\cdots	&&	\lambda_i	\\
   \end{pmatrix} \, .
   \end{equation*}   

   Given a rational matrix $A$, its Jordan Normal Form $J=QAQ^{-1}$ can be
   computed in polynomial time, as shown in \cite{Cai94}.
   Note that each vector $\myvector{v}$ appearing as a column of the
   matrix $Q^{-1}$ is a generalised eigenvector. We also note that the
   index $\nu(\lambda)$ of some eigenvalue $\lambda$ corresponds to the
   dimension of the largest Jordan block associated with it.
   Given $J_{i}$, a Jordan block of size $k$ associated with some eigenvalue $\lambda$,
   the closed-form expression for its exponential is
   \noindent
   \begin{equation*}
   \exp(J_{i}t)=\exp(\lambda t) \begin{pmatrix}
   1		&&	t		&&	\cdots	&&	\frac{t^{k-1}}{(k-1)!}	\\
   0		&&	1		&&	\cdots	&&	\frac{t^{k-2}}{(k-2)!}	\\
   \vdots	&&	\vdots	&&	\ddots	&&	\vdots						\\
   0		&&	0		&&	\cdots	&&	t							\\
   0		&&	0		&&	\cdots	&&	1							\\
   \end{pmatrix} \, .
   \end{equation*}
   
   
   Using this, for all $j\leq d$, the closed form of the $j$-th component of
   a trajectory is,
   $x^{(j)}(t) = \sum_{\lambda\in \sigma(A)} p_\lambda(t) 
   \exp(\lambda t)$ 
   where for all $\lambda\in \sigma(A)$, $p_\lambda$ is a polynomial of degree at 
   most $\nu(\lambda)-1$.
	
	
		
%
%


\subsection{The Discrete Polytope Escape Problem}

We shall also consider the Discrete Polytope Escape Problem (DPEP)\@.
The DPEP consists in deciding, given an affine function
$f:\Reals^{d}\rightarrow \Reals^{d}$ and a convex polytope
$\poly\subseteq\Reals^{d}$, whether there exists an initial point
$\myvector{x}_{0} \in \poly$ for which the sequence
$(\myvector{x}_n)_{n\in \N}$ defined by the initial point and the
recurrence $\myvector{x}_{n+1} = f(\myvector{x}_n)$ is entirely
contained in $\poly$. The definitions of fixed and trapped points are
immediately transposed to the discrete setting by considering the
sequence instead of the trajectory.
	
As with the CPEP, a $d$-dimensional instance of the DPEP is
represented by a triple $(A,B,\myvector{c})$, where $A\in\R^{d\times
  d}$ represents the function $f_A:\myvector{x} \in \R^d \mapsto
A\myvector{x}\in \R^d$ and $B\in\R^{n\times d}$ and
$\myvector{c}\in\R^{n}$ represent the polytope $\poly_{B,\myvector{c}}
= \{\myvector{x}\in \R^d \mid B\myvector{x}\leq \myvector{c}\}$.
Using the Jordan Normal form, one can see that the general form of the
$j$-th component of the sequence $(\vec{x}_n)_{n\in\N}$ is
$\myvector{x}^{(j)}_n =\sum_{\lambda\in \sigma(A)} p_\lambda(n)
\lambda^n$, where for all $\lambda\in \sigma(A)$, $p_\lambda$ is a
polynomial of degree at most $\nu(\lambda)-1$. We assume that all the coefficients of $A$, $B$ and $\myvector{c}$ are
rational.

The examples showing one cannot build a bound when the
polytope is open or unbounded for the CPEP can easily be carried over
to the DPEP. Thus, when considering the DPEP, we also only consider
compact polytopes.

\section{Deciding the Polytope Escape Problem for Compact Polytopes}
	
While the result of~\cite{OuakninePW17} allows us to decide the
existence of a trapped point for continuous linear dynamical systems,
the method is quite involved. When restricting ourselves to compact
polytopes, however, we can use the following proposition, which shows
that the existence of a trapped point is equivalent to the existence
of a fixed point.  
	
\begin{theorem}
\label{th:fixtrapC}
Given a CPEP instance $(A,B,\myvector{c})$, the polytope $\poly_{B,\myvector{c}}$ 
contains a trapped point iff it contains a fixed point.
\end{theorem}
\begin{proof}
For the ``if'' direction, observe that
a fixed point $\myvector{x}_0\in \poly_{B,\myvector{c}}$ is
necessarily trapped.

Conversely, assume that there exists a trapped point
$\myvector{x}_0\in \poly_{B,\myvector{c}}$.  Let $H$ be the closure of
the convex hull of $X(\infty)=\{ \myvector{x}(t) \mid t\in \R_{\geq 0}\}$.
Then $H$ is convex, compact, and is contained in
$\poly_{B,\myvector{c}}$.  For each $n\in \N$ we define a function
$s_n:H\rightarrow H$ by
$s_n(\myvector{x})=e^{A2^{-n}}\myvector{x}$.  Note that this function is
well-defined: clearly $X(\infty)$ is invariant under $s_n$; moreover, since
$s_n$ is linear, the convex hull of $X(\infty)$ is also invariant
under $s_n$; finally, since $s_n$ is continuous, the closure of the
convex hull of $X(\infty)$ (i.e., $H$) is invariant under $s_n$.



For all $n\in \N$, as the function $s_n$ is continuous, by Brouwer's
fixed-point theorem $s_n$  admits at least one fixed point on $H$.  Let
$F_n$ be the non-empty set of fixed points of $s_n$ in $H$. 
Since $s_n = s_{n+1}\circ s_{n+1}$ we have that $F_{n+1} \subseteq F_{n}$
for all $n\in\N$.  Moreover, by continuity of the function $f_A$, $F_n$ is a closed set
for all $n\in\N$.
Therefore, the intersection $F_\infty=\bigcap_{n\in\N} F_n$ is
non-empty.  By continuity of $f_A$, any point $\myvector{y}\in
F_\infty$ satisfies $f_A(\myvector{y})=\myvector{0}$. Therefore, the CPEP
instance admits at least one fixed point within
$\poly_{B,\myvector{c}}$, which concludes the proof.
\end{proof}

Since the set $F = \{\vec{x} \mid A\vec{x} = \vec{0}\}$ of fixed points
is easy to calculate, we simply need to check whether its intersection
with the polytope is empty in order to decide CPEP.  Since the latter
can be formulated as a linear program, we can decide CPEP for compact
polytopes in polynomial time.

The proof of Theorem~\ref{th:fixtrapC} carries over with very small changes
(considering the function $f_A$ directly, instead of the family
$(s_n)_{n\in\N}$) to prove an analogous result for DPEP: 
\begin{theorem}
Given a DPEP instance $(A,B,\myvector{c})$, $\poly_{B,\myvector{c}}$ 
has a trapped point iff it contains a fixed point.
\end{theorem}

\section{Bounding the Escape Time for a Positive CPEP Instance}
\label{sec:bound-continuous}	
	
The goal of this section is to establish a uniform bound on the escape time of a
positive CPEP instance. The main result is as follows:

\begin{theorem}
\label{th:globfin}
Given a $d$-dimensional positive instance of the CPEP, described by a tuple of bit size $b$, the time to escape the polytope is bounded by
$$T = 4 \exp\left( 640 b d^{4d+10}\right)= e^{bd^{O(d)}} .$$
\end{theorem}

We prove this bound in four steps. First, in
Subsection~\ref{sec:complex-eiger}, we show that one can ignore the component of the initial vector lying in the
complex eigenspace $\mathcal{V}^c$ after a certain amount of time.  Intuitively speaking, this stems
from the fact that a convex polytope that contains a spiral must
contain the centre of that spiral. Thus whenever we have a complex
eigenvalue we can ignore the effects of the rotation by focusing on
the axis of the helix formed by the trajectory.

We could then try to find a bound on escape time by looking at
positivity of expressions of the form $\vec{b}^T\exp(At)y_0$, where
$\vec{b}$ is normal to a hyperplane supporting a face of the
polytope. Unfortunately, these expressions contain terms corresponding
to many different eigenvalues, which significantly complicates the analysis.
We get around this problem in
Subsection~\ref{sec:hypercube} by bounding the distance of the
polytope to the origin and to the set of fixed points of the
differential equation using hypercubes in the Jordan basis. This
allows us to disentangle the effects of the different eigenvalues.  We
prove that the trajectories of the system escape the enclosing
hypercube, and use the escape time of the hypercube as an upper bound
on the escape time of the polytope.

Our next step is then, in Subsection~\ref{sec:escaping-time}, to
compute a uniform escape bound for our hypercube.  Finally,
Subsection~\ref{sec:combinesbound} combines the results from the
previous sections to get the desired bound on the escape time of the
original polytope.

\subsection{Removing the Complex Eigenvalues}
\label{sec:complex-eiger}

Let $(A,B,\myvector{c})$, be a positive CPEP instance. Assume 
for now that $A$ is given in Jordan normal form. This assumption is not without  
cost as we will see in the next subsection.
In this subsection, we consider a single block $J_i$ of $A$ corresponding to a non-real eigenvalue $\lambda_i$.
Considering only the dimensions associated to the Jordan block $J_i$ (\emph{i.e.}, the 
space $\mathcal{V}_{\lambda}$) and writing
$k= \nu(\lambda_i)$, we have that
given an initial point $\vec{x}_0=[x^{(1)},\dots,x^{(k)}]$,
the components of the trajectory $\vec{x}(t)$ are
$$
{\begin{bmatrix} 
	x^{(1)}(t)\\
	x^{(2)}(t)\\
	\vdots\\
	x^{(k)}(t) \\
	\end{bmatrix}} = \exp(\lambda_i t)\begin{bmatrix} 
x^{(1)} + x^{(2)} t + x^{(3)} t^2/2 + \dots + x^{(k)} t^{k-1}/(k-1)!\\
x^{(2)} + x^{(3)} t + \dots + x^{(k)} t^{k-2}/(k-2)!\\
\vdots\\
x^{(k)} \\
\end{bmatrix}.$$

In order to compute the escape times in the presence of non-real eigenvalues we use 
the fact that if a convex set contains a spiralling or helical trajectory, it must 
contain the axis of that trajectory. 
A trajectory starting on this axis is not affected by the
eigenvalue that generates the rotation, 
moreover, if the trajectory starting in the axis escapes, then 
the original trajectory also escapes (albeit, potentially a bit later). 
This allows us to reduce to the case where we only have real eigenvalues.
The following lemma formalizes this intuition. 

\begin{restatable}[Zero in convex hull]{lemma}{generzero} 
	\label{lem:generzero}
Let $$\vec{x}(t) = (p_{1,0}(t)e^{\lambda_1 t},\dots, p_{1,\nu(\lambda_1)-1}(t)e^{\lambda_1 t} \, , \dots , p_{r,0}(t)e^{\lambda_r t},\dots,p_{r,\nu(\lambda_r)-1}(t)e^{\lambda_r t} )^T$$ be a trajectory where, for all $j$, $\lambda_j = \eta_j + i\theta_j$, 
$\theta_j$ is non-zero, and $p_{j,k}$ is the Taylor polynomial corresponding to the factor $e^{\lambda_j t}$ of degree $k$. 
Then there exists a time $T$ such that $\Conv(X(T))$ contains the origin (where $\Conv$ represents the convex hull). In particular, this $T$ satisfies $$T \leq \sum_{j=1}^r \nu(\lambda_j)\frac{\pi}{\theta_j}.$$
\end{restatable}

\begin{proof}[Proof Sketch]
The basic idea is to take an initial point parametrized by $t$, travel along the trajectory to the point of opposite phase for a particular component, and create a new point where this component is equal to 0 by adding together a suitable convex combination of the opposite-phase point and the initial one. 
Since both these points were parametrized by $t$, we can take the trajectory starting in the newly created point (which lies in the convex hull of the original trajectory) and repeat for the other dimensions until every component corresponding to the $\mathcal{V}^c$ subspace are equal to 0. 
	
\end{proof}

\subsection{Replacing the Polytopes with Hypercubes}
\label{sec:hypercube}

Let $(A,B,\myvector{c})$ be a $d$-dimensional positive CPEP instance,
$J\in \R^{d\times d}$ a matrix in Jordan normal form, and $Q\in
\R^{d\times d}$ be such that $A = Q^{-1}J Q$.  In the first instance
we are interested in obtaining an escape-time bound in the special
case that all eigenvalues of $A$ are real.

Let us assume that all eigenvalues of $A$ are real.  Our approach is
to perform a change of basis.  To this end we note that the trajectory
$\vec{x}(t)=\exp(At)\vec{x}_0$ escapes the polytope
$\poly_{B,\myvector{c}}$ for all $\vec{x}_0\in\mathbb{R}^d$ if and
only if the trajectory $\vec{y}(t)=\exp(Jt)\vec{y}_0$ escapes the
polytope $\poly_{BQ^{-1},\myvector{c}}$ for all
$\vec{y}_0\in\mathbb{R}^d$.  (Note that all entries of $Q^{-1}$ are
real algebraic.)  Below we analyse the latter version of CPEP, i.e.,
with a matrix $J$ in Jordan form with real algebraic entries.

The key intuition is that for every initial vector $\vec{y}_0 \in
\mathbb{R}^d$ the trajectory $\vec{y}(t)=\exp(Jt)\vec{y}_0$ will
either converge to a fixed point of the system or otherwise will
diverge to infinity in length.  In either case the trajectory must exit
the polytope since the polytope is bounded and does not meet the set
$F:=\{ \vec{y} \in \mathbb{R}^d \mid J\vec{y}=\vec{0}\}$ of fixed points.
We are thus led to define constants $C,\varepsilon>0$ such that every
trajectory $\vec{y}(t)=\exp(Jt)\vec{y}_0$ that either exits the hypercube
$[-C,C]^d$ or comes within distance $\varepsilon$ of the set $F$
of fixed points will necessarily have left the polytope
$\poly_{BQ^{-1},\myvector{c}}$.  More precisely, we seek $C>0$
and $\varepsilon>0$ such that:
\begin{enumerate}
\item $\poly_{BQ^{-1},\myvector{c}} \subseteq [-C,C]^d$,
\item For all $\myvector{y} \in F$
the hypercube $\{\boldsymbol{y}+\boldsymbol{x}\mid\boldsymbol{x}\in[-\varepsilon,\varepsilon]^n\}$
does not meet $\poly_{BQ^{-1},\myvector{c}}$.
\end{enumerate}
Note that such a positive $\varepsilon$ must exist since,
$\poly_{BQ^{-1},\myvector{c}} \cap F = \emptyset$,
$\poly_{BQ^{-1},\myvector{c}}$ is compact, and $F$ is closed.  Having
computed $C$ and $\varepsilon$, we obtain the escape bound for the
polytope $\poly_{BQ^{-1},\myvector{c}}$ by computing the time to
either exit the hypercube in Item 1 or enter one of the hypercubes
mentioned in Item 2.


%



In order to compute the escape bound, we only need the upper
bound on the ratio $C/\varepsilon$ given in the following lemma.


\begin{restatable}{lemma}{hypsize} 
	\label{lem:hypsize}
Let $(A,B,\myvector{c})$, be a $d$-dimensional positive CPEP instance
involving rationals, each of at most $b\in \N$ bits. One can select
$C\in \R$ and $\varepsilon >0$ satisfying Conditions 1 and 2, above,
and such that
$$\frac C \varepsilon \leq  \exp\left ( 640 b d^{3d+8}\right ).$$
\end{restatable}

\begin{proof}[Sketch of proof.]
The proof relies on Liouville's inequality, which states that algebraic numbers can be bounded in terms of the degree and height (coefficient size) of their minimal integer polynomial, and an arithmetic complexity lemma which bounds the logarithmic height of the output of an arithmetic circuit in terms of the heights of the inputs. We analyse the computation of vertices of the polytope in the Jordan basis with these results to arrive at the final bound.
%
\end{proof}

Let us illustrate how the change of basis can lead to an exponential size polytope.
Consider the matrix
 \begin{equation*}
A= \begin{bmatrix}

	1		&&	1			&&	0		\\
	0		&&	1			&&	1		\\
	0		&&	0			&&	1.01	\\
\end{bmatrix},
\end{equation*} its associated Jordan decomposition
\begin{equation*}
A= Q^{-1}J Q=\begin{bmatrix}

	1		&&	0			&&	10000	\\
	0		&&	1			&&	100		\\
	0		&&	0			&&	1	\\
\end{bmatrix}\begin{bmatrix}

	1		&&	1			&&	0		\\
	0		&&	1			&&	0		\\
	0		&&	0			&&	1.01	\\
\end{bmatrix}\begin{bmatrix}

	1		&&	0			&&	-10000	\\
	0		&&	1			&&	-100		\\
	0		&&	0			&&	1	\\
\end{bmatrix}
\end{equation*}
and the polytope $\poly = \{(0,1,x_3)\in \R^3 \mid 0\leq x_3\leq 1\}$.
This polytope is contained in the hypercube of size $C=1$ and every point is at least
at distance $\varepsilon = 1$ from any fixed point.
However, in the Jordan basis, this  polytope becomes equal to the set 
$\{(-10000 x_3, 1-100 x_3, x_3)\in \R^3 \mid (0,1,x_3)\in \poly\}$, 
which forces a choice of $C$ and $\varepsilon$ such that $\frac C \varepsilon\geq 10 000$.

%

In general, using the same reasoning on the matrix of dimension $d$
\begin{equation*}
A = \begin{bmatrix}
1	&&	1		&&	0		&&	\cdots	&&	0		\\
0		&&	1	&&	1		&&	\cdots	&&	0		\\
\vdots	&&	\vdots	&&	\vdots	&&	\ddots	&&	\vdots	\\
0		&&	0		&&	0		&&	\cdots	&&	1		\\
0		&&	0		&&	0		&&	\cdots	&&	1 + 1/2^b	\\
\end{bmatrix},
\end{equation*}
leads to a blowup in the value for $C/\varepsilon$ of $2^{b(d-1)}$, thus exponential in the dimension.

The bound obtained in Lemma~\ref{lem:hypsize} is however doubly exponential in the 
dimension. Analysing the proof of the lemma, in order to obtain an example for which the bound is 
tight, one would need to build a family of polynomials with splitting fields of 
degree esponential in the degree of the polynomial. Such polynomials unfortunately seem
hard to find.


\subsection{Computing an Upper Bound on the Escape Time for each Eigenspace}
\label{sec:escaping-time}



Consider a real eigenvalue $\lambda$ of the Jordan matrix $J$ associated with a Jordan block of size $k$. Let $\vec{x}_0=(x^{(1)},x^{(2)},\dots,x^{(k)})$ be a point in the polytope. By 
construction of $C$, we know that $\forall i\leq k, x^{(i)}\leq C$.
The trajectory $\vec{x}(t)$, in 
that generalized eigenspace is
$$
{\begin{bmatrix} 
	x^{(1)}\\
	x^{(2)}\\
	\vdots\\
	x^{(k)} \\
	\end{bmatrix}}(t) = \exp(\lambda t)\begin{bmatrix} 
x^{(1)} + x^{(2)} t + \frac{x^{(3)} t^2}{2} + \dots + \frac{x^{(k)} t^{k-1}}{(k-1)!}\\
x^{(2)} + x^{(3)} t + \dots + \frac{x^{(k-1)} t^{k-2}}{(k-2)!}\\
\vdots\\
x^{(k)} \\
\end{bmatrix}.$$

The trajectory, limited to this Jordan block, will either escape the
hypercube $[-C,C]^d$ that encloses $\poly_{BQ^{-1},\myvector{c}}$
, or will become so small that it will be at distance
less than $\varepsilon$ from the fixed point $\vec{0}$.  We therefore
consider three cases: $\lambda=0$ and $\lambda>0$ for which the
trajectory will grow, and $\lambda<0$ which decreases the
coefficients.  Once we have an escape bound for each eigenvalue, we
will deduce a uniform bound for the entire trajectory.

Note that escaping the hypercube or converging to a fixed point do
not give symmetric results: If we find a single component that grows
larger than $C$, this is enough to escape the polytope, but
\textit{all} dimensions need to become smaller than $\varepsilon$ in
order to escape via entering the $\varepsilon$-region around the fixed
point.


\noindent {\bf{Case $\lambda<0$.}}

%
%
%
%
%

For all 
$j\leq k$, $x^{(j)} (t) = \exp(\lambda t)\sum_{i=j}^{k} x^{(i)}\frac{t^{i-j}}{(i-j)!}.$
Using the bounds on the coefficients, we thus have when $t>1$

$$|x^{(j)} (t)| = |\exp(\lambda t)\sum_{i=j}^{k} x^{(i)}\frac{t^{i-j}}{(i-j)!}| 
\leq \exp(\lambda t) kCt^k \;\text{ for } j \in \{1,\ldots,k\}$$

In order to have $|x^{(j)} (t)| < \varepsilon$, it is enough to have $\exp(\lambda t) kCt^k < \varepsilon$, which is equivalent to
$\frac{kCt^k}{\varepsilon} < \exp(-\lambda t)$, and
$  t > \frac{1}{-\lambda}\log \left (\frac{kC}{\varepsilon}\right ) + \frac{k}{-\lambda}\log t$

Here we need a small technical lemma.\\

\begin{lemma}
	[Lemma A.1 and A.2 from \cite{shalev2014understanding}]
	\label{lem:techi}
	Suppose $a \geq 1$ and $b > 0,$ then $t \geq a \log t + b$ if $t \geq 4a\log(2a) + 2b$.
\end{lemma}

Applying this lemma with $a = \max \{1, \frac{k}{-\lambda}\}$ (we assume $\frac{k}{-\lambda}>1$ in the following in order not to overload the formulas) and 
$b = \frac{1}{-\lambda}\log \left (\frac{kC}{\varepsilon}\right )$, 
we get a bound $T_\lambda$ such that for all $j\leq k$, $x^{(j)}(T)<\varepsilon$, namely 
$$\boxed{T_\lambda \leq \frac{4k}{-\lambda
}\log\left(\frac{2k}{-\lambda}\right) + \frac{2}{-\lambda}\log 
\left (\frac{kC}{\varepsilon}\right ).}$$

\noindent {\bf{Case $\lambda=0$.}}

In this case, the trajectory restricted to this eigenspace is 
$$x^{(j)} (t) = \sum_{i=j}^{k} x^{(i)}\frac{t^{i-j}}{(i-j)!} \;\text{ for } j \in \{1,\ldots,k\}.$$ 


Assume that there exists $j \geq 2$ such that $|x^{(j)}| > \varepsilon$. This holds because by definition of $\varepsilon$ a point of the polytope is at distance
at least $\varepsilon$ from a fixed point.
Moreover, $x^{(1)}$ is excluded because the line $\{x_j = 0 \mid j \neq 1\} $ is a line of fixed points of the differential equation.
Now we require a time $T_\lambda$ such that at least one of these components is larger in magnitude than $|C|$.
We construct an upper bound on this time iteratively, using the fact that at least one coefficient $x^{(j)}$ is greater than $\varepsilon$, and all of them are less than $C$, 
giving the following bound on $T_\lambda$:

$$\boxed{T_\lambda \leq \frac{1}{k} \left (\frac{k^2C}{\varepsilon}\right )^{2^{k-1}} .}$$


\noindent {\bf{Case $\lambda>0$.}}

This case proceeds similarly to the $\lambda=0$ case, although the presence of an exponential factor gives us a much better bound $T_\lambda$:
$$\boxed{T_\lambda \leq \frac{2^{k-1}}{\lambda}\log\left (\frac{kC}{\varepsilon}\right ) .}$$ 

%

\subsection{Constructing a Uniform Bound}
\label{sec:combinesbound}

We can now combine the results of the previous sections to get a uniform escape
bound, considering all eigenvalues (real or not) simultaneously. 
Let the complex eigenvalues of $A$ be $\{\eta_1 + i \theta_1,\eta_1 -
i \theta_1 \dots,\eta_{r} + i \theta_{r}, \eta_{r} - i \theta_{r}\}$
and the real eigenvalues be $\{\lambda_{1}, \dots, \lambda_{s}\}$.
Consider an arbitrary trajectory $\myvector{x}(t)$ satisfying the
differential equation $\dot{\myvector{x}}(t)=A\myvector{x}(t)$.  By
Lemma~\ref{lem:generzero} we know that for $T_c := \sum_{j=1}^r
\nu(\eta_j + i\theta_j)\frac{\pi}{\theta_j}$ there exists a point in
the convex hull of $\{ \myvector{x}(t)\mid 0\leq t \leq T_c\}$ that
lies in the real eigenspace of $A$.  This allows us to derive 
a bound on the escape time of the polytope $\poly$ from a bound on the escape time
of $\poly \cap \mathcal{V}^r$. 
Indeed, let $T_r$ be such that every "real" trajectory escapes the polytope
in time $T_r$. Then any "complex" trajectory of duration $T_c+T_r$ 
contains in its convex hull a "real" trajectory of duration $T_r$ which thus must have escaped the polytope. As the polytope is convex, this means that the 
complex trajectory itself escaped.

As for the subspace $\mathcal{V}^{r}$, we can derive from the 
escape bounds $T_\lambda$
on each eigenspace computed in Subsection~\ref{sec:escaping-time}
a time bound beyond which every real point has escaped the polytope.

\begin{lemma}[Real Time Bound]
	\label{prop:realbfin}
	Given an initial point $\vec{x}_0 \in \mathbb{R}^n$ with zero components in $\mathcal{V}^{c}$, 
	the trajectory $\vec{x}(t)$ escapes within time $T_r= 2 \max_{\lambda} T_\lambda.$ 
\end{lemma}
\begin{proof}
	Within a time $T_r/2 =  \max_{\lambda} T_\lambda $, thanks to the analysis
	of subsection~\ref{sec:escaping-time}, there are three possibilities:
	\begin{itemize}
		\item the trajectory escapes the hypercube of size $C$, this occurs if there was a coefficient associated to a non-negative eigenvalue that was larger than $\eps$;
		\item all coefficients are now smaller than $\eps$, entering the hypercube of size $\eps$ and escaping the polytope since all the purely imaginary coefficients are zero;
		\item some component corresponding to a positive or zero eigenvalue originally 
		less than $\eps$ has become greater than $\eps$. In this case, waiting another 
		$T_r/2$ amount of time puts the trajectory in the first case, ensuring it escapes.	
	\end{itemize}
	Thus in all cases the trajectory has escaped by time $T_r$.
\end{proof}

From the above, we can deduce that every trajectory escapes within time $T_r + T_c$.
We finally obtain Theorem~\ref{th:globfin} by analysing the complexity of this time bound in terms
of the number of bits of the instance and its dimension.

The magnitude of the resulting escape bound is singly exponential in the bit size of the matrix entries and doubly exponential in the dimension of the matrix. However, if the matrix is diagonalizable or invertible, we can ignore the case where the eigenvalue is zero. Then the bound becomes $O(4^{bd^2})$ which is singly exponential in the bit size and dimension.
%

In Subsection~\ref{sec:hypercube} we showed how the change of basis explained the 
exponential factor in the number of dimensions. It is clear that the escape 
time can also be exponential in the bit size of the matrix.

For a very simple example, consider a 1-dimensional case where the
polytope is the interval $[1,2]$ and the differential equation is
$\dot{x}(t)=2^{-b}x(t)$ (which obviously can be written using constants
of bit size at most $b$).  Then the initial point $x_0=1$ yields a
trajectory $x(t) = \exp(2^{-b}t)x_0$ whose escape time is $2^b\log2$,
which is exponential in $b$.

\section{The Discrete Case}
\label{sec:discretesol}

Tiwari \cite{Tiw04} and Braverman \cite{Braverman2006} have shown decidability for 
the DPEP over the rationals and reals. In general, even if every trajectory is known to be escaping, it is not possible to place a uniform bound on the number of steps.
However if the polytope is compact, we can use techniques similar to those used
for the CPEP in order to provide a bound.

\begin{theorem}
	\label{th:bounddisc}
Given a $d$-dimensional positive DPEP instance $(A,B,\myvector{c})$ where the rational
numbers use at most $b\in\N$ bits and an initial point $\vec{x}_0$, then for 
$N=e^{bd^{O(d)}}$, we have $\vec{x}_N\not\in \poly_{B,\vec{c}}$. 
\end{theorem}
\begin{proof}[Sketch of proof]

The main idea of this proof is to reduce this problem to the continuous case. 
Assuming every eigenvalue is positive, the matrix logarithm $G$ of $A$ is well defined.
The trajectory of a continuous linear dynamical sysems generated by $G$ is of the form
$\vec{x}(t) = \exp(Gt)\vec{x}(0)$.
In particular, for an initial point $x_0$ and $n \in \Naturals$, we have
\[\vec{x}(n) = \exp(G n)\vec{x}_0 = \exp(G)^n \vec{x}_0 = A^n \vec{x}_0 = \vec{x}_n\]
Therefore, we can relate the escape time of the CPEP instance 
$(G,B,\myvector{c})$ to the escape time of the DPEP instance $(A,B,\myvector{c})$.

The eigenvalues that are not positive are dealt with using a variant of the convex hull Lemma~\ref{lem:generzero}.
%
\end{proof}


\bibliography{refs}

\appendix

\section{Proof of Section~\ref{sec:bound-continuous}}\label{appendix:undec}

\subsection{Proof of Lemma~\ref{lem:generzero}}

We establish this result by induction over $r$, the number of distinct eigenvalues.

%

\subparagraph{Base case.}

Assume $r=1$, we have 
$$\vec{x}(t) = e^{\eta_1 t}e^{i\theta_1 t}(p_{1,0}(t),p_{1,1}(t),\dots, p_{1,\nu(1)-1}(t)) \in \mathbb{C}^{\nu(1)}.$$
We define a new starting point belonging to the convex hull of the trajectory $\vec{x}(t)$ by
$$\vec{z}_{1}(0) = \frac{p_{1,0}(\frac{\pi}{\theta_1})e^{\eta_1 \frac{\pi}{\theta_1}  }}{p_{1,0}(0)+p_{1,0}(\frac{\pi}{\theta_1})e^{\eta_1 \frac{\pi}{\theta_1}  }}\vec{x}(0) +\frac{p_{1,0}(0)}{p_{1,0}(0)+p_{1,0}(\frac{\pi}{\theta_1})e^{\eta_1 \frac{\pi}{\theta_1}  }}\vec{x}(\frac{\pi}{\theta_1}).$$


Now observe that 
\begin{align*}
z_{1}(0) &= \frac{p_{1,0}(\frac{\pi}{\theta_1})e^{\eta_1 \frac{\pi}{\theta_1}  }}{p_{1,0}(0)+p_{1,0}(\frac{\pi}{\theta_1})e^{\eta_1 \frac{\pi}{\theta_1}  }}(  \, p_{1,0}(0),p_{1,1}(0),\dots, p_{1,\nu(1)-1}(0)  \,)\\ 
&+\frac{p_{1,0}(0)}{p_{1,0}(0)+p_{1,0}(\frac{\pi}{\theta_1})e^{\eta_1 \frac{\pi}{\theta_1}  }}e^{\eta_1 \frac{\pi}{\theta_1}  }e^{i\theta_1 (\frac{\pi}{\theta_1})}(  \, p_{1,0}(\frac{\pi}{\theta_1}),\dots, p_{1,\nu(1)-1}(\frac{\pi}{\theta_1})  \,)\\
&=\frac{e^{\eta_1 \frac{\pi}{\theta_1}  }}{p_{1,0}(0)+p_{1,0}(\frac{\pi}{\theta_1})e^{\eta_1 \frac{\pi}{\theta_1}  }}
(0,\dots, p_{1,0}(\frac{\pi}{\theta_1})p_{1,\nu(1)-1}(0) - p_{1,0}(0)p_{1,\nu(1)-1}(\frac{\pi}{\theta_1}))\\
&=(0, q_{1,0}(0),\dots, q_{1,\nu(1) - 2}(0)),
\end{align*}
where $q_{1,k}(t)$ is a polynomial of degree at most $k$. 



Iterating this process, we build the family of points $(\vec{z}_{k})_{k\leq \nu(\lambda_1)}$ such that the $k$ first coordinates of $\vec{z}_{k}$ are null. Thus,
we have $\vec{z}_{\nu(\lambda_1)}(t) = \vec{0} \in \mathbb{C}^{\nu(\lambda_1)}.$
Each step of this process requires an additional $\frac{\pi}{\theta_1}$ time units to ensure the constructed point belongs to the convex hull of the trajectory $\vec{x}(t)$. Thus after $T \geq \nu(\lambda_1)\frac{\pi}{\theta_1},$ we have $\vec{0} \in \Conv (X(T))$ as required.

\subparagraph{Inductive case.} Let $r\geq 1$ and 
$$\vec{x}(t) = (p_{1,0}(t)e^{\lambda_1 t},\dots, p_{1,\nu(\lambda_1)-1}(t)e^{\lambda_1 t} \, , \dots , p_{r+1,0}(t)e^{\lambda_{r+1} t},\dots,p_{r+1,\nu(\lambda_{r+1})-1}(t)e^{\lambda_{r+1} t} )^T.$$
By induction hypothesis, for $T_1=\sum_{j=1}^r \nu(\lambda_j)\frac{\pi}{\theta_j}$, there exists a point $\vec{z}_0$ in $\Conv(X(T_1))$ such that the components corresponding to the first $r$ eigenvalues remain equal to 0. Therefore, the trajectory starting in 
$\vec{z}_0$ is of the form 
$$\vec{z}(t) = (0,\dots, 0, q_{r+1,0}(t)e^{\lambda_{r+1} t},\dots,q_{r+1,\nu(\lambda_{r+1})-1}(t)e^{\lambda_{r+1} t} )^T$$
where the $q_{r,j}$ are polynomials.

Applying the process used in the base case, one gets that the zero vector belongs to
the set $\Conv(Z(T_2))$ for $T_2 = \nu(\lambda_{r+1}) \frac{\pi}{\theta_{r+1}}$.
Moreover, as $\vec{z}_0\in \Conv(X(T_1))$, we have that $\vec{0} \in \Conv(X(T_1+T_2))$.	

\subsection{Proof of Lemma~\ref{lem:hypsize}}


We first recall some known results on the heights of algebraic numbers. More details can be found in~\cite{Waldschmidt2000}.\\

\begin{defi}[Naive height]
	Given an algebraic number $\alpha$, its \emph{naive height} $H(\alpha)$  is the largest absolute value of any coefficient of its minimal polynomial in the ring $\mathbb{Z}[x]$. The \emph{degree} of $\alpha$ is the degree of this polynomial.\\
\end{defi}

\begin{lemma}[Liouville's inequality]
\label{lem:Liouv}
	Given $\alpha \neq 0$ an algebraic number, $$\frac{1}{H(\alpha)+1} < |\alpha| < H(\alpha)+1.$$
\end{lemma}

One can also define a logarithmic height of an algebraic number. It satisfies the following lemma.

\begin{lemma}[logarithmic height]
	Given an algebraic number $\alpha$ of degree $n$, its \emph{absolute logarithmic height} $h(\alpha)$ satisfies the following relations.
	$$\frac{1}{n} \log H(\alpha) - \log 2 < h(\alpha) < \frac{1}{n} \log H(\alpha) +\frac{1}{2n} \log (n+1)$$
	
	Moreover, for algebraic numbers $\alpha_1, \alpha_2, \dots, \alpha_k$, we have
	
	$$h(\prod_{i=1}^k \alpha_i) \leq \sum_{i=1}^k h(\alpha_i), \mbox{ and }
	h(\sum_{i=1}^k \alpha_i) \leq \log k + \sum_{i=1}^k h(\alpha_i).$$
	
\end{lemma}

This lemma directly implies the following result.\\

\begin{corollary}[arithmetic complexity]
	Given algebraic numbers $\{\alpha_i\}_{i=1,\dots, k}$, such that $h(\alpha_i) \leq h_{max}$ and a function $f$ that computes an arithmetic circuit involving at most $m$ operations of addition, multiplication, subtraction and division,
	
	then $$\displaystyle h\left( f(\alpha_1,\dots,\alpha_k) \right) \leq (m+1)h_{max} + m\log 2.$$
\end{corollary}

%
%
%
%

We aim to give an upper bound on the ratio $C/\eps$. If we can find a maximum 
height $H_{max}$ of any component of a vertex of the polytope in the Jordan basis, 
then, using Lemma~\ref{lem:Liouv}, 
$(H_{max}+1)$ is an upper bound of the (component-wise) distance of any point of the polytope to the origin, and $1/(H_{max}+1)$ is a lower bound. Thus $(H_{max}+1)^2$ is an upper bound for $C/\eps$.

Given a vector $\vec{v}$, let $|\vec{v}|_\infty$ denote the max norm (largest absolute value of any component of $\vec{v}$).

Recall that $J=QAQ^{-1}$ is the Jordan normal form of $A$, so we work in the basis $y = Qx$. Note that each vector $\myvector{v}$ appearing as a column of the
matrix $Q^{-1}$ is a generalised eigenvector.

Let $V$ be the set of vertices of the original polytope, namely vectors of the form $B'^{-1}c'$ for $B'$ invertible square submatrix of $B$.
Then we can select $C$ and $\eps$ such that
$$\frac{C}{\eps} = \frac{\max_{x \in V} |Qx|_\infty}{\min_{x \in V} |Qx|_\infty}.$$

Let us give a bound of $|Qx|_\infty$ for some $x\in \poly_{B,\vec{c}}$. 
For all $i\leq d,$ we have 
 $h((Qx)^{(i)}) = h(\sum_j q_{ij}x^{(j)}) \leq 2d( \max_{i,j\leq d}
\{h(q_{ij}), h(x^{(i)})  \} +\log 2 ).$

$h(x_i)$ is easy to compute.
The elements $x_i$ are solutions to the linear system $B'x = c$.
Using Gaussian elimination, which has less than $3d^3$ arithmetic complexity,
we can compute the entries of $x$.
Since the entries of $B$ and $c$ are at most $b$-bit rationals, they have logarithmic height $b$. 
This implies that $h(x_i) \leq 3d^3 (b+ \log 2)$.

The entries $q_{ij}$ of $Q$ can be computed in $3d^3$  operations from the entries of $Q^{-1}$, which is the matrix of generalised eigenvectors of $A$.
For $q$ and $q^{-1}$ representing the maximum logarithmic height of an entry of $Q$ and 
$Q^{-1}$ respectively, via Gaussian elimination, we have:	
\begin{equation}
q \leq 3d^3 (q^{-1} + \log 2)
\end{equation}
Let $\lambda$ be an eigenvalue of $A$, and let $M_\lambda = A - \lambda I$.
Then the columns of $Q^{-1}$ are vectors $v$, where $v$ is a generalized eigenvector that satisfies $M_\lambda^d v =0$
 for some eigenvalue $\lambda$.
Note that the equation $M_\lambda^d v =0$ is underdetermined, but this does not matter for our purposes, since we just need one valid eigenvector to compute a bound on the height. Note that by the definition of Jordan normal form, $M_\lambda^d v =0$ has at least one non-zero solution.
Again by Gaussian elimination, for $m_\lambda^d$ the maximum logarithmic height of an entry of $M_\lambda^d$, we have:	
\begin{equation}
q^{-1} \leq 3d^3 (m_\lambda^d + \log 2)
\end{equation}
Computing an element of $M_\lambda^d$ from $M_\lambda$ is more complicated and gives height:
\begin{equation}
m_\lambda^d \leq (2d)^d(m_\lambda + \log d)
\end{equation}
Since $M_\lambda = A - \lambda I$, for $a$ the maximum logarithmic height of an entry of $A$,
\begin{equation}
m_\lambda \leq a + h(\lambda) + \log 2 \leq b + h(\lambda) + \log 2
\end{equation}

The height of any eigenvalue $\lambda$ is determined by the coefficients of the characteristic polynomial of $A$. For this, we rely on the following lemma given in~\cite{dumas2007bounds}.

\begin{lemma}[characteristic polynomial bound]
	Let $A \in \Complex^{n \times n}$ , with $n \geq 4$, whose coefficients are bounded in absolute value by
	$B > 1$. The coefficients of the characteristic polynomial $C_A$ of $A$ are denoted by $c_j , j = 0, . . . , n$.
	and $||C_A||_\infty = \max_j\{|c_j |\}$. 
	
	Then $||C_A||_\infty \leq (2nB^2)^{n/2}$. 
	
	Note this is only a factor of $2^{n/2 }$ larger than the Hadamard bound on the determinant, which is the zeroth coefficient.
\end{lemma}

Since the entries of $A$ are $b$-bit rationals, we can use $2^b$ as the bound $B$ of this lemma. 
Thus the coefficients of the characteristic polynomial over $\Rationals$ are bounded by $(2d2^{2b})^{d/2}$.
Multiplying by $2^{bd}$ (the largest possible denominator) ensures an integer polynomial,
we obtain the following bound on $h(\lambda)$:
\begin{equation}
h(\lambda) \leq \log_2  [(2d2^{2b})^{d/2}2^{bd}] \leq 2bd + \frac{d}{2} \log (2d)\leq 3bd^2.
\end{equation}

Gathering (1), (2), (3), (4) and (5), we obtain:
%
%
$$q \leq 80bd^{d+7}2^d.$$
Thus, for all $i\leq d,$ we have 
$$h((Qx)^{(i)})\leq 2d( \max\{h(q_{ij}), h(x_i)  \}+\log 2 ) \leq 160bd^{d+8}2^d.$$
Using the log height lemma, we get the bound
$$H((Qx)^{(i)}) \leq 2^n \exp(n h((Qx)^{(i)})),$$ 
where $n$ is the degree of $(Qx)^{(i)}$ as an algebraic number.

%

Since $(Qx)^{(i)}$ is obtained by performing arithmetic operations with all the eigenvalues of $A$, it lies in the splitting field of the characteristic polynomial of $A$, which may have degree $d!$. 
Using Lemma~\ref{lem:Liouv}, this gives us a quantitative bound on $C/\eps$, which is
$$C/\eps \leq (H((Qx)^{(i)}) +1)^2 \leq 4H((Qx)^{(i)})^2 \leq 
2^{2d!+2} \exp\left ( 320 b d^{d+8} 2^d  (d!)\right ) \leq 
\exp\left ( 640 b d^{3d+8}\right ) .$$

\subsection{Proofs of Section~\ref{sec:escaping-time}}

\noindent {\bf{Case $\lambda=0$.}}

\begin{proof}
	Given the set of equations $$x_j (t) = \sum_{i=j}^{k} x_{i}\frac{t^{i-j}}{(i-j)!},$$ we want a $T$ such that there exists $j$ such that $|x_j(T)| > C$.
	
	By construction of $\eps$, there exists $j_1 \geq 2$ such that $|x_{j_1}| > \eps$.
	Consider the component $$x_{j_1 -1} (t) = \sum_{i={j_1 -1}}^{k} x_{i}\frac{t^{i-j_1 +1}}{(i-j_1 +1)!}.$$
	We set $T_{j_1} = kC/\eps$.
	Observe that $$|x_{j_1 -1} (T_{j_1})| \geq |x_{j_1} kC/\eps | - |x_{j_1 -1}| - \sum_{i={j_1 +1}}^{k} \left |x_{i}\frac{T_{j_1}^{i-j_1 +1}}{(i-j_1 +1)!}\right |$$
	Since the first term is larger than $kC$ and the second term is smaller than $C$, the only way $|x_{j_1 -1} (T_{j_1})|$ could be less than $C$ (and thus not escape the polytope) is if one of the later terms is larger than $C$. 
	Let $j_2$ be the highest index such that $\left |x_{j_2}\frac{T_{j_1}^{{j_2}-j_1 +1}}{({j_2}-j_1 +1)!}\right |\geq C.$ Note that $j_2 > j_1$.
	We now have a lower bound on a higher index coefficient, namely
	$|x_{j_2}|\frac{T_{j_1}^{j_2-j_1 +1}}{(j_2-j_1 +1)!} > C$.
	We now repeat the process with the component $$x_{j_2 -1} (t) = x_{j_2 -1}+ x_{j_2}t + \sum_{i={j_2 +1}}^{k} x_{i}\frac{t^{i-j_2 +1}}{(i-j_2 +1)!}$$
	We have $|x_{j_2}|\frac{T_{j_1}^{j_2-j_1 +1}}{(j_2-j_1 +1)!} > C,$ thus setting $T_{j_2} > k\frac{T_{j_1}^{j_2-j_1 +1}}{(j_2-j_1 +1)!}$ ensures that $|x_{j_2}T_{j_2}| > kC$.
	
	
	Continuing this process, we will either find a component that escapes the polytope
	or move on to a component with higher index, which can happen at most $k-1$ times, because we have the constraints $j_1 \geq 2, \forall i, j_i \leq k, $ and $j_i > j_{i-1}$.
	This gives us a recursive definition for the bound, which is $$T_{j_n} > k\frac{T_{j_{n-1}}^{j_n-j_{n-1} +1}}{(j_n-j_{n-1} +1)!} $$
	
	We wish to find an upper bound on $T = T_N$, the time by which we are guaranteed that 
	at least one component escapes, subject to the constraints $j_1 \geq 2, j_N \leq k, $ and $j_n > j_{n-1}$.
	We can solve the recursive inequality by weakening it (since we only need an upper bound on $T$) to $$T_{j_n} > (kT_{j_{n-1}})^{j_n-j_{n-1} +1}.$$
	Note that pulling the constant $k$ into the exponentiated part is valid because $j_n-j_{n-1} +1	>2$ always.
	Setting $S_{j_n} = kT_{j_n}$, we get
	$S_{j_n} > S_{j_{n-1}}^{j_n-j_{n-1} +1}, S_{j_1}= k^2C/\eps,$	
	which reduces to $$S_{j_N} > \left (\frac{k^2C}{\eps}\right )^{\prod_{i=2}^{N} (j_i -j_{i-1} + 1)}$$
	The term $\prod_{i=2}^{N} (j_i -j_{i-1} + 1)$ is maximised when for all $i$, $j_i = j_{i-1} +1$, thus in the worst case we have 
	$$S_{j_N} > \left (\frac{k^2C}{\eps}\right )^{2^{k-1}}$$
%
	Thus we have a bound for a zero-eigenvalue component to escape, which is 
	$$\boxed{T \leq \frac{1}{k} \left (\frac{k^2C}{\eps}\right )^{2^{k-1}} .}$$
\end{proof}

\subparagraph{Case $\lambda>0$.}

\begin{proof}
	The proof is very similar in structure to the zero eigenvalue case, though the presence of an exponential factor gives us a much better bound.

	By construction of $\eps$, there exists $j_1 \geq 2$ such that $|x_{j_1}| > \eps$.
	Consider the component $$x_{j_1} (t) = \sum_{i={x_1}}^{k} \exp(\lambda t)x_{i}\frac{t^{i-j_1 }}{(i-j_1)!}.$$
	Set $T_{j_1} = \frac{1}{\lambda}\log(kC/\eps)$ and observe that $$|x_{j_1} (T_{j_1})| \geq \exp(\lambda \frac{1}{\lambda}\log(kC/\eps))|x_{j_1}| - \sum_{i={j_1 +1}}^{k}  \exp(\lambda T_{j_1})|x_{i}|\frac{T_{j_1}^{i-j_1}}{(i-j_1)!}.$$
	Since the first term is larger than $kC$, the only way $|x_{j_1} (T_{j_1})|$ can be less than $C$ (and thus not escape the polytope) is if one of the later terms is larger than $C$. Let $j_2$ be the highest index such that 
	$\exp(\lambda T_{j_1})|x_{j_2}|\frac{T_{j_1}^{{j_2}-j_1}}{({j_2}-j_1)!}\geq C.$ Note that $j_2 > j_1$.
	We now have a lower bound on a higher index coefficient, namely
	$|x_{j_2}|\frac{T_{j_1}^{j_2-j_1}}{(j_2-j_1)!}\exp(\lambda T_{j_1}) > C.$
	Now we repeat the process with the component $$x_{j_2} (t) = \exp(\lambda t)x_{j_2}+ \sum_{i={j_2 +1}}^{k} \exp(\lambda t)x_{i}\frac{t^{i-j_2 }}{(i-j_2)!}$$
	We want $|x_{j_2}| \exp(\lambda T_{j_2} ) > kC$, so it is enough to set 
	$$\exp(\lambda T_{j_2} )\frac{(j_2-j_1)!}{T_{j_1}^{j_2-j_1}}\exp(-\lambda T_{j_1}) > k .$$
	Ignoring the factorial term for simplicity, we get the constraint $$T_{j_2} > T_{j_1} + \frac{j_2-j_1}{\lambda}\log(T_{j_1}) + \frac{1}{\lambda}\log k.$$
	
	Continuing the process, we will either find a component that escapes the polytope
	or move on to a component with higher index, which can happen at most $k-1$ times, because we have the constraints $j_1 \geq 2, \forall i, j_i \leq k, $ and $j_i > j_{i-1}$.
	This process gives us a recursive definition for the bound, which is $$T_{j_n} > T_{j_{n-1}} + \frac{j_n-j_{n-1}}{\lambda}\log(T_{j_{n-1}}) + \frac{1}{\lambda}\log k.$$
	
%
%
	We wish to find an upper bound on $T = T_N$, the time by which we are guaranteed that at least one component escapes, subject to the constraints $j_1 \geq 1, j_N \leq k, $ and $j_n > j_{n-1}$.
	We can solve the recursive inequality by weakening it (since we only need an upper bound on $T$), observing that $T_{j_{n-1}} >\frac{k}{\lambda}\log(T_{j_{n-1}}) \Rightarrow T_{j_{n-1}} > \frac{j_n-j_{n-1}}{\lambda}\log(T_{j_{n-1}})$
	and the lefthand side of this implication holds if $T_{j_{n-1}} > \frac{4k}{\lambda}\log \left (\frac{2k}{\lambda}\right )$ (using Lemma~\ref{lem:techi}).
	Thus if we ensure that $T_{j_{1}} > \frac{4k}{\lambda}\log \left (\frac{2k}{\lambda}\right )$, we can work with the much simpler recurrence 
	$$T_{j_n} > 2T_{j_{n-1}} +  \frac{1}{\lambda}\log k ,$$
	which is easily solved to get 
	$$T_N > 2^{N-1}T_{j_1} - \frac{1}{\lambda}\log k.$$
	
	As $N \leq k$ and assuming $\frac{1}{\lambda}\log(kC/\eps)> \frac{4k}{\lambda}\log \left (\frac{2k}{\lambda}\right )$ (which is valid as the order of magnitude of the first one is greater than the second one), 
	we have a bound for a positive-eigenvalue component to escape, which is 
	$$\boxed{T \leq \frac{2^{k-1}}{\lambda}\log\left(\frac{kC} \eps\right) .}$$ 
\end{proof}

\subsection{Analysis of the complexity of $T_c+T_r$}

The escape time is bounded by 
\begin{align*}
T_c+T_r \leq & \sum_{j=1}^r \nu(\eta_j + i\theta_j)\frac{\pi}{\theta_j}  +\\&
2 \max_{\lambda} \left \{ \frac{4\nu(\lambda)}{|\lambda|}\log\frac{2\nu(\lambda)}{|\lambda|} + \frac{2}{|\lambda|}\log \frac{\nu(\lambda)C}{\eps}, \;
 \frac{2^{\nu(\lambda)-1}}{\lambda}\log\frac{\nu(\lambda)C}\eps, \; \frac 1  {\nu(0)}\left (\frac{\nu(0)^2C}{\eps}\right )^{2^{\nu(0)-1}} \right \}.
\end{align*}

In terms of magnitude, the worst case occurs for a zero eigenvalue, in which case 
$ T_r \leq \frac 2  {\nu(0)}(\nu(0)^2C/\eps)^{2^\nu(0)}\leq \frac 2  {d}(d^2C/\eps)^{2^d}$.

We need to bound $\frac C \epsilon, \frac 1 \theta,$ and $\frac 1 \lambda$.
The bound on $\frac C \epsilon$ was given in Lemma~\ref{lem:hypsize}:
$$\exp\left ( 640. b. d^{3d+8}\right ).$$

For $\frac{1}{\theta}$ and $\frac{1}{\lambda}$, 
we can use the Mignotte root separation bound :
\begin{prop} [Mignotte bound]
	Let $f\in\mathbb{Z}[x]$. If $\alpha_{1}$ and $\alpha_{2}$ are distinct roots of $f$, then
	\begin{align*}
	\lvert \alpha_{1}-\alpha_{2} \rvert > \frac{\sqrt{6}}{d^{(d+1)/2}H^{d-1}}
	\end{align*}
	where $d$ and $H$ are respectively the degree and height (maximum
	absolute value of the coefficients) of $f$.
\end{prop}
We apply this result on the polynomial $xP(x)$ where $P$ is the characteristic polynomial
of $A$. This gives a distance between the root $0$ and $\lambda$. It also gives a bound for $\theta$ as it is obtained as the difference between two conjugate roots of $A$.
Using the bounds on the height of $P$ computed for Lemma~\ref{lem:hypsize}, we get
$$\frac{1}{\theta}, \frac{1}{\lambda} \leq \frac 1 {\sqrt{6}} \left ((2d)^{d/2}\cdot 2^{2bd}\right )^{d-1} d^{(d+1)/2} \leq 4^{3bd^3}.$$

Thus, $T_c$ is very small compared to the bound obtained for $T_r$ in the zero eigenvalue case. Therefore, the escape time is bounded by
$$T \leq 
\frac 4  {d}(d^2\exp\left ( 640 b d^{3d+8}\right ))^{2^d}\leq 
4 \exp\left ( 640 b d^{4d+10}\right )= e^{bd^{(d)}}.$$

\section{Proofs of Theorem~\ref{th:bounddisc}}

Given a $d$-dimensional instance $(A,B,\vec{c})$ of the DPEP, the proof is realised in two steps. 
First we deal with the negative and zero eigenvalues, showing how they can be ignored
in a similar way as we did for the continuous case in Subsection~\ref{sec:complex-eiger}
(\emph{i.e.} by showing that the axe of the symmetries created by the negative eigenvalues is in the polytope and if the trajectories starting on the axis escape, then every trajectory escapes in a small additional number of steps).
Then, we will reduce the problem, with only non-negative eigenvalues, to the continuous 
case using the matrix logarithm.


Let us start with the negative eigenvalues.
We state a lemma that is essentially a discrete version of Lemma~\ref{lem:generzero}.
\begin{lemma}[Zero in convex hull (discrete case)]
	Suppose there are $r$ negative real eigenvalues, $\lambda_1,\dots, \lambda_r$ with $\lambda_j$ of multiplicity $\nu(j)$. Let $N = \sum_{j=1}^r \nu(j)$.
	
	Given a vector $v$ where every component is equal to 0 outside the negative real eigenspaces, the convex hull of the set $\{ v, Av, A^2v , \dots, A^Nv\}$ contains the origin.
\end{lemma}
\begin{proof}
	The point $A^nv \in \R^N$ (considering only the negative real eigenspace, since other coordinates are zero) can be written as $$\vec{x}_n = (p_{1,0}(n)\lambda_1^n,\dots, p_{1,\nu(1)-1}(n)\lambda_1^n  \, , \, p_{2,0}(n)\lambda_2^n ,\dots,p_{r,\nu(r)-1}(n)\lambda_r^n  )^T,$$ where $p_{j,k}$ is a polynomial corresponding to the $j$'th eigenvalue of degree $k$, as can be seen from the Jordan normal form.
	
	Now observe that $\lambda_j^n = (-|\lambda_j|)^n = e^{( \log |\lambda_j| + i \pi )n}.$
	Thus we can apply the method of Lemma~\ref{lem:generzero}, with each step being done in one time unit, giving a bound $N \leq \sum_{j=1}^r \nu(j)$ on the number of 
	steps to contain the origin.
\end{proof}

We observe that this bound on the number of steps is always smaller or equal to the dimension of the instance.
For the eigenspaces associated to zero eigenvalues, note that the corresponding
Jordan blocks are nilpotent, so in at most $d$ iterations,
the component in this eigenspace goes to zero.


Now, let us assume the components associated to negative or zero eigenvalues are null.
Let $J$ be the Jordan Normal Form of $A$.
Let $\mathcal{V}_-$ be the eigenspace of negative and zero eigenvalues, and let $\mathcal{V}_c$ be the eigenspace of all other eigenvalues.
Let $J' $ be the submatrix of $J$ restricted to $\mathcal{V}_c$. 
Similarly, define $\mathcal{P'}$ to be the polytope obtained by projecting 
$\poly_{BQ^{-1},\vec{c}}$ onto this subspace. 
We define $G' = \Log J'$. As there are no negative and zero eigenvalue, this operation is well defined. $G'$ translates $J'$ into the continuous setting, in particular, 
for $n \in \Naturals$, we have (for $\vec{y}_n$ the sequence $\vec{y}_n$ in the Jordan basis and restricted to the real eigenspaces)
$\vec{y}_n = (J')^n \vec{y}_0 = \exp(G')^n \vec{y}_0
=\exp(G' n)\vec{y}_0 = \vec{y}(n).$
Let $T_c$ be the escape time bound given by Theorem~\ref{th:globfin} for the
instance defined by the polytope $\mathcal{P'}$ and the function $x\mapsto G'x$.
In a similar fashion to the analysis of the continuous case, we have that the number of steps needed to escape the polytope is bounded by
$$N = \lceil T_c \rceil + d.$$


This value is dominated by the term $T_c$, for which the worst case is due to the 
Jordan blocks with eigenvalue 1, since this corresponds to the zero eigenvalue case 
after we take the logarithm of the matrix. 
Thus we obtain that $N=e^{bd^{O(d)}}$, as in the continuous case.

\end{document}